\documentclass[11pt]{amsproc}
 \usepackage[margin=1in]{geometry}
\usepackage{setspace,fullpage}
\usepackage{graphicx}
\usepackage[nice]{nicefrac}
\usepackage{amssymb}
\usepackage{amsmath,amsthm,amscd,amssymb, mathrsfs}

\usepackage{latexsym}
\usepackage[colorlinks,citecolor=red,pagebackref,hypertexnames=false]{hyperref}

\numberwithin{equation}{section}

\theoremstyle{plain}
\newtheorem{theorem}{Theorem}[section]
\newtheorem{lemma}[theorem]{Lemma}
\newtheorem{corollary}[theorem]{Corollary}

\newtheorem{conjecture}[theorem]{Conjecture}
\newtheorem{question}[theorem]{Question}

\theoremstyle{definition}

\theoremstyle{remark}
\newtheorem{remark}[theorem]{Remark}

\newtheorem{case[theorem]}{Case}

\title[\parbox{14cm}{\centering{ Harmonic analysis related to homogeneous varieties \hspace{1in}}} \quad]{ Harmonic analysis related to homogeneous varieties in three dimensional vector space over finite fields }
\author{ Doowon Koh and Chun-Yen Shen }

\address{Department of Mathematics\\
Michigan State University \\
East Lansing, MI 48824,  USA}
\email{koh@math.msu.edu}

\address{Department of Applied Mathematics\\
 National Chiao Tung University\\
Hsinchu 300, Taiwan}
\email{chunyshen@gmail.com}



\begin{document}

\begin{abstract} In this paper we study  extension problems, averaging problems, and generalized Erd\H os-Falconer distance problems associated with arbitrary homogeneous varieties in three dimensional vector space over finite fields. In the case when homogeneous varieties in three dimension do not contain any plane passing through the origin, we obtain the general best possible results on aforementioned three problems. In particular, our results on extension problems  recover and generalize the work due to Mockenhaupt and Tao who completed the particular conical extension problems in three dimension. Investigating the Fourier decay on homogeneous varieties, we give the complete mapping properties of averaging operators over homogeneous varieties in three dimension. In addition, studying the generalized Erd\H os-Falconer distance problems related to homogeneous varieties in three dimensions, we improve the cardinality condition on sets where the size of distance sets is nontrivial. Finally, we address a question of our problems for homogeneous varieties in higher odd dimensions.
\end{abstract}
\maketitle
\tableofcontents

\section{Introduction}
In classical harmonic analysis, extension theorems and averaging problems are the problems to deal with the boundedness of operators.
In the Euclidean space, these problems have been well studied, but have not yet been solved in higher dimensions or general setting.
On the other hand, the Falconer distance problem is considered as a continuous analogue of the Erd\H os distance problem. Although these topics are origin from studying the geometry of sets, the  harmonic analysis methods have been used as a main tool. For example, the best known results on the Falconer distance problem were obtained by Erdo\~{g}an \cite{Er05} who applied Tao's bilinear restriction theorem \cite{Ta03}, one of the most beautiful theorems in harmonic analysis. However, this problem is still open in any dimensions and it has been believed that the problem could be completely solved by using the skills from other mathematical fields along with known ideas from classical harmonic analysis. For this reason, main topics in harmonic analysis  have been recently studied in the finite field setting, in part because finite fields not only serve as a typical model for the Euclidean space but also have powerful structures which enable us to  relate our problems to other well-studied problems in  arithmetic combinatorics, algebraic geometry, and analytic number theory. Moreover, finite fields often yield new facts in which one may be interested. In this paper we study such problems on harmonic analysis in finite fields. More precisely, we focus on studying the finite field analogues of the following well-known Euclidean problems related to homogeneous varieties in three dimension: the extension problem, the averaging problem, and the Erd\H os-Falconer distance problem.  \\

Before we introduce our main problems and results, let us briefly review such problems in the Euclidean setting.
Let $H$ be a set in ${\mathbb R}^d$ and $d\sigma$  a measure on the set $H.$
In the Euclidean case, the extension problem is to determine the optimal range of exponents $1\leq p,r \leq \infty$ such that the following extension estimate holds:
$$\|(fd\sigma)^\vee\|_{L^r({\mathbb R}^d)}\leq C(p,r,d)\|f\|_{L^p(H, d\sigma)} ~~\mbox{for all}~~ f\in L^p(H, d\sigma)$$
where   $(fd\sigma)^\vee$ denotes the inverse Fourier transform of the measure $fd\sigma.$
This problem was first addressed in 1967 by Stein \cite{St78} and it has been extensively studied in the last few decades.
For a comprehensive survey of this problem, we refer readers to \cite{Ta04}.\\

 The averaging problem also asks us to find the exponents $1\leq p,r \leq \infty$ such that the following inequality holds:
\begin{equation}\label{Averagingquestion}
\|f\ast d\sigma\|_{L^r({\mathbb R}^d)}\leq C(p,r,d) \|f\|_{L^p({\mathbb R}^d)}~~\mbox{for all}~~f\in L^p({\mathbb R}^d),\end{equation}
where $d\sigma$ is a measure supported on a surface $H$ in ${\mathbb R}^d$ and the convolution $f\ast d\sigma$ is defined by the relation
$ f\ast d\sigma(x)= \int_{H} f(x-y) ~ d\sigma(y)$ for  $x\in \mathbb R^d.$ For classical results of this problem, see \cite{St70}, \cite{St71}, and \cite{Li73}. In particular, Iosevich and Sawyer \cite{IS96} obtained the sharp mapping properties of averaging operators on a graph of homogeneous function of degree $\geq 2.$ \\

The Erd\H os distance problem and the Falconer distance problem are problems to measure sizes of distance sets determined by discrete sets and continuous sets respectively. Given $E, G \subset \mathbb R^d, d\geq 2,$ the distance set $\Delta(E,G)$ is defined by
$$ \Delta(E,G)=\{|x-y|: x\in E, y\in G\},$$
where $|\cdot|$ denotes the usual Euclidean norm. Given finite sets $E,G$, the Erd\H os distance problem is to determine the smallest possible cardinality of the  distance set $\Delta(E,G) $ in terms of sizes of sets $E, G.$
In the case when $E=G$, Erd\H os \cite{Er46} first studied this problem and conjectured that for every finite set $E\subset \mathbb R^d,$
$$ |\Delta(E,E)|\gtrapprox |E|^{\frac{2}{d}},$$
where $|\cdot|$ denotes the cardinality of the finite set. However, this problem has not been solved in all dimensions $d\geq 2.$ For the recent development and the best known results on this problem, see \cite{SolTo01},\cite{KT04}, \cite{SV04}, and \cite{SV05}.
As a continuous analog of the Erd\H os distance problem,  Falconer \cite{Fa85} conjectured that if the Hausdorff dimension of a Borel subset $E$ in $\mathbb R^d, d\geq 2,$ is greater than $d/2$, then  the Lebesgue measure of the distance set $\Delta(E,E)$ must be positive. This problem is known as the Falconer distance problem and is also open in all dimensions. The best known result on this problem is due to Erdo\~{g}an \cite{Er05} who extended the work \cite{Wo99} by Wolff showing that any Borel set $E$ with the Hausdorff dimension greater than $ d/2+ 1/3$ yields  the distance set $\Delta(E,E)$ with a positive Lebesgue measure.\\

The purpose of this paper is to obtain the sharp results on the extension problem, the averaging problem, and the Erd\H os Falconer distances problem associated with arbitrary homogeneous varieties in three dimensional vector space over finite fields. In the finite field setting, we shall prove that all these problems can be completely understood by observing that  most homogeneous varieties in three dimension intersect with only few lines in any plane passing through the origin. From this observation and properties of homogeneous varieties, we obtain an extremely good Fourier decay on the general homogeneous varieties in three dimension, which may be a specific property from finite fields.  It is also known \cite{Ko09} that the estimates of the Fourier transform over homogeneous varieties associated with non-degenerate quadratic polynomials are distinguished in even dimensions and odd dimensions. Precisely, obtaining good estimates of the Fourier transform over homogeneous varieties in even dimensions can not be expected. On the other hand, we shall give a reasonable conjecture to show that most homogeneous varieties in odd dimensions would yield good Fourier decay. Our work on homogeneous varieties in three dimension was mainly motivated by the conjecture. From our results, it will be clear that the conjecture  holds in three dimension. In addition, our results are of a generalization of the previously known facts on the homogeneous varieties of degree two.

\section{Notation, definitions, and  key lemmas}
Let $\mathbb F_q$ be a finite field with $q$ elements. We denote by $\mathbb F_q^d, d\geq 2, $ the $d$-dimensional vector space over the finite field $\mathbb F_q.$ Given a set $E\subset\mathbb F_q^d,$ we denote by $|E|$ the cardinality of the set $E.$ For nonnegative real numbers $A, B,$ we write $A\lesssim B$ if $A\leq CB $ for some $C>0$ independent of the size of the underlying finite field $\mathbb F_q.$ In other words, the constant $C>0$ is independent of the parameter $q.$ We also use  $A\sim B$ to indicate $A\lesssim B\lesssim A.$ We say that a polynomial $P(x)\in \mathbb F_q[x_1,\dots,x_d]$ is a homogeneous polynomial of degree $k$ if the polynomial's  monomials with nonzero coefficients all have the same total degree $k.$ For example, $P(x_1,x_2,x_3)= x_1^5+x_2^3x_3^2$ is a homogeneous polynomial of degree five. Given a homogeneous polynomial $P(x)\in \mathbb F_q[x_1,\dots,x_d],$ we define a homogeneous variety $H$ in $\mathbb F_q^d$ by the set
$$ H=\{x\in \mathbb F_q^d: P(x)=0\}.$$ For example, the cone in three dimension, which was introduced in \cite{MT04}, is a homogeneous variety generated by the homogeneous polynomial $P(x)=x_1^2-x_2x_3.$
We now review the Fourier transform of a function defined on $\mathbb F_q^d.$
Denote by $\chi$ the nontrivial additive character of $\mathbb F_q.$ For example, if $q$ is prime, then we may take $\chi(t)=e^{2\pi i t/q}$ where we identify $t\in \mathbb F_q$ with a usual integer. We now endow the space $\mathbb F_q^d$ with a normalized counting measure $dx.$
Thus, given a complex valued function $f:\mathbb F_q^d \rightarrow \mathbb C,$ the Fourier transform of $f$ is defined by
$$ \widehat{f}(m)=\int_{ \mathbb F_q^d} \chi(-m\cdot x) f(x) ~dx= \frac{1}{q^d} \sum_{x\in \mathbb F_q^d} \chi(-m\cdot x) f(x),$$
where $m$ is any element in the dual space of $(\mathbb F_q^d, dx).$ Recall that the Fourier transform $\widehat{f}$ is actually defined on the dual space of  $(\mathbb F_q^d, dx).$ We shall endow the dual space of $(\mathbb F_q^d, dx)$ with a counting measure $dm.$ We write $(\mathbb F_q^d, dm)$ for the dual space of $(\mathbb F_q^d, dx).$ Then, we also see that the Fourier inversion theorem says that for every $x\in (\mathbb F_q^d, dx),$
\begin{equation}\label{Finversion} f(x)=\int_{\mathbb F_q^d}\chi(x\cdot m) \widehat{f}(m)~dm= \sum_{m\in \mathbb F_q^d} \chi(x\cdot m) \widehat{f}(m).\end{equation}
We also recall the Plancherel theorem: $\|\widehat{f}\|_{L^2(\mathbb F_q^d,dm)}= \|f\|_{L^2(\mathbb F_q^d,dx)},$ which is same as
$$ \sum_{m\in \mathbb F_q^d} |\widehat{f}(m)|^2 = \frac{1}{q^d} \sum_{x\in \mathbb F_q^d} |f(x)|^2.$$
For instance, if $f$ is a characteristic function on the subset $E$ of $\mathbb F_q^d,$ then  the Plancherel theorem yields
\begin{equation}\label{Plancherel} \sum_{m\in \mathbb F_q^d} |\widehat{E}(m)|^2 = \frac{|E|}{q^d},\end{equation}
here, and throughout the paper, we identify the set $E \subset \mathbb F_q^d$ with the characteristic function on the set $E$, and we denotes by $|E|$ the cardinality of the set $E \subset \mathbb F_q^d.$

\begin{remark}\label{convension} For a simple notation, we use the notation $\mathbb F_q^d$ for both the space and its dual space. However, this may make readers confused, because the measures of the space and its dual space are different. To overcome this confusion, we always use the variable $``x"$ as an element of the space $(\mathbb F_q^d, dx)$ with the normalized counting measure $dx.$ For example, we write $x\in \mathbb F_q^d$ for $x\in (\mathbb F_q^d, dx).$ On the other hand, we always use the variable $``m"$ as an element of the dual space $(\mathbb F_q^d, dm)$ with a counting measure $dm.$
Thus, $m\in \mathbb F_q^d$ means that $m\in (\mathbb F_q^d, dm).$
\end{remark}

\subsection{Fourier decay on homogeneous varieties}
We shall estimate the Fourier transform of  characteristic functions on homogeneous varieties in three dimensional vector space over the finite field $\mathbb F_q.$
First, let us review the well-known Schwartz-Zippel lemma, which gives us the information about the cardinality of any variety in $\mathbb F_q^d.$ For a nice proof of the Schwartz-Zippel lemma below, see Theorem $6.13$ in \cite{LN93}.
\begin{lemma} [Schwartz-Zippel] Let $P(x)\in \mathbb F_q[x_1,\cdots,x_d]$ be a nonzero polynomial of degree $k.$ Then, we have
$$|\{x\in \mathbb F_q^d: P(x)=0\}|\leq k q^{d-1}.$$
\end{lemma}
Using the Schwartz-Zippel lemma, we obtain the following lemma.
\begin{lemma}\label{intersection}
Given a nonzero homogeneous polynomial $P(x)\in \mathbb F_q[x_1,x_2,x_3]$, let $H$ be the homogeneous variety given by
$$ H=\{x\in \mathbb F_q^3: P(x)=0\}.$$
If the homogeneous variety $H$ does not contain any plane passing through the origin, then we have
for every $m\in \mathbb F_q^3\setminus \{(0,0,0)\}, $
$$ |H\cap \Pi_m|\lesssim q,$$
where $\Pi_m=\{ x\in \mathbb F_q^3: m\cdot x=0\}$ which is a hyperplane passing through the origin.
\end{lemma}
\begin{proof}
First, let us observe the set of $H\cap \Pi_m.$
Fix $m\neq (0,0,0).$ Without loss of generality, we may assume that $m=(m_1,m_2, -1).$ Then, we see that
$$ \Pi_m=\{x\in \mathbb F_q^3: m_1x_1+m_2x_2 - x_3=0\}.$$
and
$$ H=\{x\in \mathbb F_q^3: P(x)=0\}.$$
Thus, we see that
$$H\cap \Pi_m=\{(x_1,x_2, m_1x_1+m_2x_2)\in \mathbb F_q^3: P(x_1,x_2, m_1x_1+m_2x_2)=0\}.$$
Put $R(x_1,x_2)=P(x_1,x_2, m_1x_1+m_2x_2).$ Then, it is clear that
$$ |H\cap \Pi_m|=| \{(x_1,x_2)\in \mathbb F_q^2: R(x_1,x_2)=0\}|.$$
If $R(x_1,x_2)$ is a nonzero polynomial, then the Schwartz-Zippel lemma tells us that $ |H\cap \Pi_m|\lesssim q$ and we complete the proof.
Now assume $R(x_1,x_2)$ is a zero polynomial.
Then, it follows that $R(x_1,x_2)=P(x_1,x_2, m_1x_1+m_2x_2)=0$ for all $x_1,x_2 \in \mathbb F_q.$ This implies that the variety $H=\{x\in \mathbb F_q^3: P(x)=0\}$ contains a plane
$ m_1x_1+m_2x_2-x_3=0,$ which contradicts to our hypothesis that $ H$ does not contain any plane passing through the origin. Thus, the proof is complete.
\end{proof}

In order to compute the Fourier transform on homogeneous varieties, we shall need the following Lemma \ref{formulaFourier}. We remark that the proof of Lemma \ref{formulaFourier} below adopts the invariant property of homogeneous varieties which was already observed before (see \cite{Co94}). For readers' convenience, we state the lemma in a slightly different way and give an explicit proof here.
\begin{lemma}\label{formulaFourier} Let $P(x)\in \mathbb F_q[x_1,\dots,x_d]$ be a nonzero homogeneous polynomial.
Define a homogeneous variety $H \subset \mathbb F_q^d$ by
$$H=\{x\in \mathbb F_q^d: P(x)=0\}.$$
For each $m\in \mathbb F_q^d,$ we have
\begin{equation}\label{sample1} \widehat{H}(m)= \frac{1}{q^{d+1}-q^d} \left( q|H\cap \Pi_m|-|H|\right),\end{equation}
where $\Pi_m=\{x\in \mathbb F_q^d: m\cdot x=0\}.$\end{lemma}

\begin{proof} For each $m\in \mathbb F_q^d,$ we have
$$ \widehat{H}(m)=q^{-d} \sum_{x\in  H} \chi(-m\cdot x).$$
Since $P(x)$ is a homogeneous polynomial, a change of the variable yields that for each $t\neq 0$,
$$ \widehat{H}(m)=\widehat{H}(tm).$$
It therefore follows that
$$\widehat{H}(m)= q^{-d}(q-1)^{-1}\sum_{x\in H}\sum_{t\in \mathbb F_q\setminus \{0\}} \chi(-tm\cdot x)$$
$$= q^{-d}(q-1)^{-1}\sum_{x\in H}\sum_{t\in \mathbb F_q} \chi(-tm\cdot x)- q^{-d}(q-1)^{-1} |H|.$$
By the orthogonality relation of nontrivial additive character $\chi$, we complete the proof.
\end{proof}

From lemma \ref{intersection} and \ref{formulaFourier}, the Fourier transform on homogeneous varieties in three dimension can be estimated. The following corollary shall make a crucial role for proving our results in this paper.
\begin{corollary}\label{Maincor} Suppose  the homogeneous variety
$ H=\{x\in \mathbb F_q^3: P(x)=0\}$ does not contain any plane passing through the origin in $\mathbb F_q^3,$
where $P(x)$ is a homogeneous polynomial in $\mathbb F_q[x_1,x_2,x_3].$
Then, for any $m\neq (0,0,0),$ we have
\begin{equation}\label{decayvariety} |\widehat{H}(m)|=\left| \frac{1}{q^3}\sum_{x\in H} \chi(-m\cdot x)\right|\lesssim q^{-2}.\end{equation}
\end{corollary}
\begin{proof} Since the homogeneous variety $H$ does not contain any plane passing through the origin, it is clear that the polynomial $P(x)\in \mathbb F_q[x_1,x_2,x_3]$ is a nonzero polynomial. Thus, the Schwartz-Zippel lemma says that $|H|\lesssim q^2$ and so Corollary \ref{Maincor} follows immediately from Lemma \ref{formulaFourier} and Lemma \ref{intersection}.
\end{proof}
\begin{remark}\label{inversionremark} Let $d\sigma$ be the normalized surface measure of the homogeneous variety $H\subset \mathbb F_q^3$ given in Corollary \ref{Maincor}.
Then, we notice that if $|H|\sim q^2,$ then the conclusion $(\ref{decayvariety})$ in Corollary \ref{Maincor} implies that for every $m\in \mathbb F_q^3\setminus \{(0,0,0)\},$
\begin{equation} \label{inversiondecay}|(d\sigma)^\vee(m)|=\left| \frac{1}{|H|} \sum_{x\in H} \chi(m\cdot x)\right|\lesssim q^{-1}.\end{equation}
\end{remark}

\section{Extension problems for finite fields}
In the finite field setting, Mockenhaupt and Tao \cite{MT04} first set up and studied the extension problem for various algebraic varieties.
Here, we review the definition of the extension problem for finite fields and introduce our main result on the problem for homogeneous varieties in three dimension. For a fixed polynomial $P(x)\in \mathbb F_q[x_1,\cdots,x_d],$  consider an algebraic variety $V$ in $\mathbb F_q^d, d\geq 2, $ given by  $$ V=\{x\in \mathbb F_q^d: P(x)=0\}.$$
From Remark \ref{convension}, recall that  the variety $V$ is considered as a subset of the space $(\mathbb F_q^d, dx)$ with the normalized counting measure $dx.$ Therefore, if $f: (\mathbb F_q^d, dx) \rightarrow \mathbb C$ is a complex valued function, then for $1\leq p< \infty$ the $L^p-$norm of $f$ takes the following value:
$$ \|f\|_{L^p(\mathbb F_q^d, dx)}= \left( \frac{1}{q^d} \sum_{x\in \mathbb F_q^d} |f(x)|^p \right)^{\frac{1}{p}}.$$
As usual, $\|f\|_{L^\infty (\mathbb F_q^d, dx)}$ is the maximum value of $|f|.$ We now endow the variety $V$ with the normalized surface measure $\sigma$ such that the total mass of $V$ is one. In other words, the surface measure $\sigma$ supported on $V$ can be defined by the relation
\begin{equation}\label{alex} d\sigma(x)= \frac{q^d}{|V|} V(x)~ dx,\end{equation}
here, recall that we identify the set $V \subset \mathbb F_q^d$ with the characteristic function $\chi_{V}$ on the set $V$.
Thus, we see that
$$ \|f\|_{L^p(V, \sigma)} =\left(\int_{ V} |f(x)|^p d\sigma(x)\right)^{\frac{1}{p}}= \left( \frac{1}{|V|} \sum_{x\in V} |f(x)|^p \right)^{\frac{1}{p}},$$
and the inverse Fourier transform of measure $fd\sigma$ is given by
$$ (fd\sigma)^{\vee}(m)=\int_{V} \chi(m\cdot x) f(x)~d\sigma(x)= \frac{1}{|V|} \sum_{x\in V} \chi(m\cdot x) f(x),$$
where we recall that $m$ is an element of the dual space $(\mathbb F_q^d, dm)$ with the counting measure $dm.$
In addition, note that for $1\leq p,r<\infty,$
$$\| (fd\sigma)^\vee \|_{L^r(\mathbb F_q^d, dm)} = \left(\int_{\mathbb F_q^d} | (fd\sigma)^\vee(m)|^r dm\right)^{\frac{1}{r}}
= \left( \sum_{m\in \mathbb F_q^d} |(fd\sigma)^\vee(m)|^r \right)^{\frac{1}{r}}$$
and $\| (fd\sigma)^\vee \|_{L^\infty(\mathbb F_q^d, dm)}$ takes the maximum value of $ |(fd\sigma)^\vee|.$
\subsection{Definition of the extension theorem}
Let $1\leq p,r\leq \infty.$ We denote by $R^*(p\to r)$ to be the smallest constant such that for all functions $f$ on $V,$
$$ \| (fd\sigma)^\vee \|_{L^r(\mathbb F_q^d, dm)} \leq R^*(p\to r) \|f\|_{L^p(V, d\sigma)}.$$
By duality, we note that the quantity $R^*(p\to r)$ is also the best constant such that the following restriction estimate holds: for every function $g$ on $({\mathbb F_q^d},dm),$
\begin{equation}\label{restrictiondef}
\|\widehat{g}\|_{L^{p'}(V,d\sigma)}\leq R^*(p\to r) \|g\|_{L^{r'}({\mathbb F_q^d}, dm)},
\end{equation}
where $p'$ and $r'$ denote the dual exponents of $p$ and $r$ respectively, which mean that  $1/p+1/p'=1$ and $1/r+1/r'=1.$

Observe that $R^*(p\to r)$ is always a finite number but it may depend on the parameter $q,$ the size of the underlying finite field $\mathbb F_q.$
In the finite field setting, the extension problem is to determine the exponents $1\leq p,r\leq \infty$ such that
$$ R^*(p\to r)\leq C,$$
where the constant $C>0$ is independent of $q.$ A direct calculation yields the trivial estimate, $R^*(1\to \infty)\lesssim 1.$
Using H\"{o}lder's inequality and the nesting properties of $L^p$-norms, we also see that
$$ R^*(p_1\to r)\leq R^*(p_2\to r)\quad\mbox{for}~~1\leq p_2\leq p_1\leq \infty$$
and
$$ R^*(p\to r_1)\leq R^*(p\to r_2)\quad \mbox{for}~~ 1\leq r_2\leq r_1\leq \infty.$$
In order to obtain the strong result on the restriction problem, for each $1\leq p (\mbox{or}~r) \leq \infty$, we shall try to find the smallest number $1\leq r~(\mbox{or}~ p)~ \leq \infty$ such that
$R^*(p\to r) \lesssim 1.$ In addition, using the interpolation theorem, it therefore suffices to find the critical exponents $1\leq p, r\leq \infty.$
\subsection{Necessary conditions for $R^*(p\to r) \lesssim 1$}
In \cite{MT04}, Mockenhaupt and Tao showed that if $|V|\sim q^{d-1}$ and the variety $V\subset \mathbb F_q^d$ contains
an $\alpha$-dimensional affine subspace $\Pi (|\Pi|=q^\alpha),$ then  the necessary conditions for $R^*(p\to r)\lesssim 1$ are given by
\begin{equation}\label{necessary1} r\geq \frac{2d}{d-1}\quad\mbox{and}~~
r\geq \frac{p(d-\alpha)}{(p-1)(d-1-\alpha)}.\end{equation}
Now, let us consider the homogeneous variety $H=\{x\in \mathbb F_q^3: P(x)=0\}$ in three dimension where $P(x)\in \mathbb F_q[x_1,x_2,x_3]$ is a homogeneous polynomial. In addition, assume that $|H|\sim q^2.$ It is clear that the homogeneous variety $H$ contains a line, because if $P(x_0)=0$ for some $x_0\neq (0,0,0)$, then $P(tx_0)=0$ for all $t\in \mathbb F_q.$ From (\ref{necessary1}), we therefore see that the necessary conditions for $R^*(p\to r)\lesssim 1$ take the following:
$$r\geq 3 \quad\mbox{and}~~
r\geq \frac{2p}{p-1}.$$
In particular, if $H=\{x\in \mathbb F_q^3: x_1^2-x_2x_3=0\}$ which is a cone in three dimension, then above necessary conditions can be improved by the conditions:
$$ r\geq 4 \quad\mbox{and} ~~ r\geq \frac{2p}{p-1}.$$
This was proved by Mockenhaupt and Tao (see Proposition $7.1$ in \cite{MT04}). Moreover, they proved that $R^*(2\to 4)\lesssim 1$ which implies that the necessary conditions are in fact sufficient conditions. Thus, the $L^2-L^4$ extension estimate can be considered as the generally best possible result for the extension problems related to arbitrary homogeneous varieties in three dimensions. It is unknown whether there exists a homogeneous variety in
$\mathbb F_q^3$ to yield the better extension estimates than the conical extension estimates. However, it is easy to see that there exists a homogeneous variety on which the best possible extension estimates are worser than the conical extension estimates. For example, if we take a homogeneous variety  as
$H=\{x\in \mathbb F_q^3: x_1+x_2+x_3=0\},$
then $H$ contains a plane, two dimensional subspace, and so the necessary conditions in (\ref{necessary1}) say that only trivial $L^p-L^\infty, 1\leq p\leq \infty,$ estimates hold. From this example, one may ask which homogeneous varieties in $\mathbb F_q^3$ yield as good extension estimates  as the conical extension estimates? In the following section, our main result on the extension problems will answer this question.
\section{Main result on the extension problems}
We prove that if the homogeneous variety in $\mathbb F_q^3$ does not contain any plane passing through the origin, then the extension estimates are as good as  the conical extension estimates. More precisely, we have the following main result.
\begin{theorem}\label{mainextensionthm} For each homogeneous polynomial $P(x)\in \mathbb F_q[x_1,x_2,x_3]$, let $H=\{x\in \mathbb F_q^3: P(x)=0\}.$
Suppose that $|H|\sim q^2$ and the homogeneous variety $H$ does not contain any plane passing through the origin. Then, we have the following extension estimate on $H$:
$$ R^*(2\to 4)\lesssim 1.$$
\end{theorem}

We shall give two different proofs of Theorem \ref{mainextensionthm}. One is based on geometric approach and the other is given in view of the Fourier decay on the homogeneous variety in three dimension.
\subsection{The proof of Theorem \ref{mainextensionthm} based on geometric approach}
In order to show that $R^*(2\to 4)\lesssim 1,$ we shall use the following well-known lemma for reduction, which is basically to estimating the incidences between the variety and its nontrivial translations.
For a complete proof of the following lemma, see both Lemma $5.1$ in \cite{MT04} and Lemma 13 in \cite{IK08}.
\begin{lemma}\label{L2L4formula} Let $V$ be any algebraic variety in $\mathbb F_q^d, d\geq 2,$ with $|V|\sim q^{d-1}.$
Suppose that for every $\xi\in \mathbb F_q^d\setminus \{(0,\dots,0)\},$
$$ \sum_{(x,y)\in V\times V: x+y=\xi } 1 \lesssim q^{d-2}.$$
Then, we have
$$ R^*(2\to 4) \lesssim 1.$$
\end{lemma}

Using  Lemma \ref{L2L4formula}, the following lemma shall give the complete proof of Theorem \ref{mainextensionthm}.

\begin{lemma} Let $P(x)\in \mathbb F_q[x_1,x_2,x_3]$ be a homogeneous polynomial.
Suppose that the homogeneous variety $H=\{x\in \mathbb F_q^3: P(x)=0\}$ does not contain any plane passing through the origin. Then, we have that for every $\xi \in \mathbb F_q^3\setminus \{(0,0,0)\},$
$$|\{(x,y)\in H\times H: x+y=\xi\}|\lesssim q.$$
\end{lemma}
\begin{proof}
The first observation is that since $H$ is a homogeneous variety, $H$ is exactly the union of lines passing through the origin.
To see this, just note that if $P(x)=0$ for some $x\neq (0,0,0),$ then $P(tx)=0$ for all $t\in \mathbb F_q.$
Therefore, we can write
\begin{equation}\label{Obser1} H=\cup_{j=1}^N L_j,\end{equation}
where $N$ is a fixed positive integer, $L_j$ denotes a line passing through the origin, and $L_i \cap L_j=\{(0,0,0)\}$ for $i\neq j.$
From the Schwartz-Zippel lemma, it is clear that $|H|\lesssim q^2.$ Thus, the number of lines, denoted by $N$, is $\lesssim q,$ because each line contains $q$ elements. The second important observation is that if $H$ does not contain any plane passing through the origin, then for every $m\in \mathbb F_q^3\setminus \{(0,0,0)\},$
\begin{equation}\label{Obser2} |H\cap \Pi_m|\lesssim q,\end{equation}
where $\Pi_m=\{x\in \mathbb F_q^3: m\cdot x=0\}.$ This observation follows from Lemma \ref{intersection}.
We are ready to prove our lemma. For each $\xi\neq (0,0,0),$ it suffices to prove that
the number of common solutions of $P(x)=0$ and $P(\xi-x)=0$ is $\lesssim q.$ Since $P(x)$ is a homogeneous polynomial, we see that
$P(\xi-x)=0$ if and only if $P(x-\xi)=0.$ Therefore, we aim to show that for every $\xi\neq (0,0,0),$
$$ |H \cap (H+\xi)|\lesssim q,$$
where $H+\xi=\{(x+\xi)\in \mathbb F_q^3: x\in H\}.$ Now, fix $\xi\neq (0,0,0).$ From (\ref{Obser1}), we see that
$$ |H \cap (H+\xi)|\leq \sum_{j=1}^N |H\cap (L_j+\xi)|.$$
Notice that if $\xi\in L_j,$ then $L_j+\xi=L_j$ and so $|H\cap (L_j+\xi)|=q.$
However, there is at most one line $L_j$ such that $\xi\in L_j.$ Thus, it is enough to show that if $\xi\notin L_j$, then $|H\cap (L_j+\xi)|\lesssim 1$, because $N\lesssim q.$ However, this will be clear from (\ref{Obser2}). To see this, first notice that  if $(0,0,0)\neq \xi \notin L_j$, then the line $L_j+\xi$ does not pass through the origin, because the line $L_j$ passes through the origin. Thus, the line $L_j+\xi$ is different from all lines $L_k$ in $H=\cup_{k=1}^N L_k,$ and so there is at most one intersection point of the line $L_j+\xi$ and each line in $H.$  Next, consider the unique plane $\Pi_m$ which contains the line $L_j+\xi.$ Then, (\ref{Obser2}) implies that  at most few lines in $H$ lie in the plane $\Pi_m$ containing the line $L_j+\xi.$ Thus, we conclude that $|H\cap (L_j+\xi)|\lesssim 1 $ for $\xi \notin L_j.$ Thus, the proof is complete.
\end{proof}

\subsection{Remark on the Fourier decay on homogeneous varieties}
In \cite{MT04}, Mockenhaupt and Tao showed that if $d\sigma$ is the normalized surface measure on the paraboloid $V=\{x\in \mathbb F_q^d: x_d=x_1^2+\dots+x_{d-1}^2\}, d\geq 2,$ then the sharp Fourier decay of $d\sigma$ is given by
\begin{equation} \label{verygood}|(d\sigma)^\vee(m)|\lesssim q^{-\frac{d-1}{2}} \quad \mbox{for} ~~ m\in \mathbb F_q^d\setminus \{(0,\dots,0)\}.\end{equation}
From this good Fourier decay and  the Tomas-Stein type argument for finite fields, they observed that
$$ R^*\left(2\to \frac{2d+2}{d-1}\right)\lesssim 1,$$
where the exponents $p=2, r= (2d+2)/(d-1)$ are called the standard Tomas-Stein exponents.
In particular, If $d=3$, then $R^*(2\to 4)\lesssim 1$ which is exactly same as the conclusion of Theorem \ref{mainextensionthm}.
Since (\ref{inversiondecay}) in Remark \ref{inversionremark} says that the surface measure on our homogeneous variety in three dimension yields
a good Fourier decay as in (\ref{verygood}), it is not surprising that the Tomas-Stein type argument gives the complete proof of Theorem \ref{mainextensionthm}. However, it has been believed that the surface measure on homogeneous varieties in general dimensions yields the worse Fourier decay than the decay as in (\ref{verygood}). In fact, if $H=\{x\in \mathbb F_q^d : x_1^2+\dots+x_d^2=0\}$ and $d\geq 2$ is even, then using the explicit Gauss sum estimates one can show that if $m_1^2+\cdots+m_d^2=0,$ then
$$ |(d\sigma)^\vee(m)|\sim q^{-\frac{d-2}{2}}.$$
On the other hand, if $d\geq 3$ is odd, then it has been proved in \cite{Ko09} that the Fourier decay of the surface measure on $H$ is as good as the decay in (\ref{verygood}). From these facts and the good Fourier decay on our homogeneous varieties in $\mathbb F_q^3$, it seems that most homogeneous varieties in odd dimensions yield the good Fourier decay given in $(\ref{verygood}).$ In the last section, we shall give a question concerning this issue.

\subsection{The proof of Theorem \ref{mainextensionthm} by the Fourier decay on homogeneous varieties}
In the previous subsection, we have seen that the Tomas-Stein type argument for finite fields will yield the alternative proof of Theorem \ref{mainextensionthm}. For the sake of completeness, we give the complete proof. Let $R^*: L^p(H,d\sigma) \to L^r({\mathbb F}_q^3, dm)$ be the extension map  $f\to (fd\sigma)^\vee,$ and $R: L^{r^\prime}({\mathbb F}_q^3, dm)\to L^{p^\prime}(H,d\sigma)$ be its dual, the restriction map
$g\to \widehat{g}|_H.$ Observe that $ R^*Rg = (\widehat{g}d\sigma)^\vee = g \ast (d\sigma)^\vee$ for every function $g$ on $(\mathbb F_q^3,dm).$
Now, in order to prove Theorem \ref{mainextensionthm} we must show that for every function $f$ on $(H,d\sigma)$,
$$\|(fd\sigma)^\vee\|_{L^4(\mathbb F_q^3, dm)} \lesssim \|f\|_{L^2(H, d\sigma)},$$
where $d\sigma$ is the normalized surface measure on the homogeneous variety $H \subset \mathbb F_q^3.$
By duality (\ref{restrictiondef}), it is enough to show that the following restriction estimate holds:
for every function $g$ defined on $({\mathbb F_q^3},dm),$ we have
$$ \|\widehat{g}\|^2_{L^2(H,d\sigma)}\lesssim \|g\|^2_{L^{\frac{4}{3}}(\mathbb F_q^3, dm)}.$$
By the orthogonality principle, we see that
\begin{align*}\|\widehat{g}\|^2_{L^2(H,d\sigma)} =&<Rg,~Rg>_{L^2(H,d\sigma)} =<R^*Rg,~g>_{L^2({\mathbb F}_q^3, dm)}\\
=&<g \ast (d\sigma)^\vee ,~g>_{L^2({\mathbb F}_q^3, dm)}~\leq \|g \ast (d\sigma)^\vee \|_{L^4(\mathbb F_q^3, dm)}~\|g\|_{L^{\frac{4}{3}}(\mathbb F_q^3, dm)},\end{align*}
where the inequality follows by H\"older's inequality. It therefore suffices to show that for every function $g$ on $({\mathbb F_q^3},dm),$
$$\|g\ast (d\sigma)^\vee \|_{L^{4}(\mathbb F_q^3, dm)}\lesssim \|g\|_{L^{\frac{4}{3}}({\mathbb F_q^3}, dm)}.$$
For each $m\in (\mathbb F_q^3, dm),$ define $ K(m)=(d\sigma)^\vee(m) -\delta_0(m)$ where $\delta_0(m)=0$ for $m\neq (0,0,0)$ and $\delta_0(0,0,0)=1.$
Since $(d\sigma)^\vee(0,0,0)=1$, we see that $ K(m)=0$ if $m=(0,0,0)$ and $K(m)=(d\sigma)^\vee(m)$ if $m\neq (0,0,0).$ It follows that
$$ \|g\ast (d\sigma)^\vee \|_{L^{4}(\mathbb F_q^3, dm)}=\|g\ast (K+\delta_0)\|_{L^{4}(\mathbb F_q^3, dm)}\leq \|g\ast K \|_{L^{4}(\mathbb F_q^3, dm)}
+ \|g\ast \delta_0 \|_{L^{4}(\mathbb F_q^3, dm)}.$$
Since $g\ast \delta_0= g$ and $dm$ is the counting measure,  we see that
$$\begin{array}{ll} \|g\ast \delta_0 \|_{L^4(\mathbb F_q^3, dm)}&=\|g \|_{L^4({\mathbb F_q^3}, dm)}\\
&\leq \|g\|_{L^{\frac{4}{3}}(\mathbb F_q^3, dm)}\end{array}.$$
Thus, it is enough to show that for every $g$ on $({\mathbb F_q^3},dm),$
$$\|g\ast K \|_{L^4({\mathbb F_q^3}, dm)}\lesssim \|g\|_{L^{\frac{4}{3}}({\mathbb F_q^3}, dm)}.$$
However, this estimate follows immediately by interpolating the following two inequalities:
\begin{equation}\label{ltwo} \|g\ast K\|_{L^2({\mathbb F_q^3}, dm)}\lesssim q \|g\|_{L^2({\mathbb F_q^3},dm)}
\end{equation}
and
\begin{equation}\label{linfty}\|g\ast K\|_{L^\infty({\mathbb F_q^3}, dm)}\lesssim q^{-1} \|g\|_{L^1({\mathbb F_q^3},dm)}.
\end{equation}
Thus, it remains to show that both (\ref{ltwo}) and (\ref{linfty}) hold. Using the Plancherel theorem, the inequality (\ref{ltwo}) follows from the following observation:
$$ \begin{array}{ll} \|g\ast K\|_{L^2(\mathbb F_q^3, dm)}&=\|\widehat{g}\widehat{K}\|_{L^2(\mathbb F_q^3, dx)}\\
&\leq \|\widehat{K}\|_{L^\infty({\mathbb F_q^3}, dx)} \|\widehat{g}\|_{L^2({\mathbb F_q^3}, dx)}\\
&\lesssim q\|g\|_{L^2({\mathbb F_q^3}, dm)},\end{array}$$
where the last line is based on the observation that for each $x\in ({\mathbb F_q^3},dx)$\\
$\widehat{K}(x)= d\sigma(x)-\widehat{\delta_0}(x)=q^3|H|^{-1} H(x)-1 \lesssim q,$ because $|H|\sim q^2$ and $\delta_0$ is a function on $(\mathbb F_q^3, dm) $ with a counting measure $dm.$ Finally, the estimate (\ref{linfty}) follows from Young's inequality and  the Fourier decay estimate (\ref{inversiondecay}) in Remark \ref{inversionremark}. Thus, the proof is complete.

\section{Averaging problems for finite fields}

In the finite field setting, Carbery, Stones and Wright \cite{CSW08} recently addressed the averaging problems over algebraic varieties related to vector-valued polynomials. Recall that $(\mathbb F_q^d, dx), d\geq 2,$ is the $d$-dimensional vector space with the normalized counting measure $dx.$
For $1\leq k\leq d-1$, they considered a specific vector-valued polynomial $P_k: \mathbb F_q^k\to \mathbb F_q^d$ given by
$$ P_k(x)=\left(x_1,x_2,\dots,x_k, x_1^2+x_2^2+\dots+x_k^2, x_1^3+\cdots+x_k^3, \dots, x_1^{d-k+1}+\cdots +x_k^{d-k+1} \right)$$
and studied the averaging problem over the $k$-dimensional surface $V_k=\{ P_k(x)\in \mathbb F_q^d: x\in \mathbb F_q^k\}.$ Using the Weil's theorem \cite{We48} for exponential sums, they obtained the sharp, good Fourier decay on the surface $V_k,$ which led to the complete solution for the averaging problem. It will be also interesting to study the averaging problem over some algebraic varieties which can not be explicitly defined by a vector-valued polynomial. Koh \cite{Ko09} studied the averaging problem over the variety $V=\{x\in \mathbb F_q^d: a_1x_1^2+a_2x_2^2+\cdots+a_dx_d^2=0\}$ for all $a_j\neq 0.$ Using the explicit Gauss sum estimates, he observed that if the dimension $d$ is odd, then
the sharp Fourier decay on the variety $V$ is given by $|\widehat{V}(m)|\lesssim q^{-(d+1)/2}$ for all $m\neq (0,\dots,0).$ In addition, he showed that
if the dimension $d\geq 3$ is odd, then the complete solution for the averaging problem over the variety $V$  can be obtained by simply applying the good Fourier decay on $V.$ However, when the dimension $d\geq 2$ is even, it was also observed by Koh that the sharp Fourier decay on $V$ takes the following worse form: $ |\widehat{V}(m)|\lesssim q^{-d/2}$ for every $m\neq (0,\dots,0)$  and so the averaging problem becomes much harder. From Koh's observations, one may guess that most homogeneous varieties in odd dimensions yield the good Fourier decay but the homogeneous varieties in even dimensions do not. Authors in this paper do not know the exact answer for this issue. However, our Corollary \ref{Maincor} gives the positive answer for the homogeneous varieties in three dimensions. In this section, we shall show that Corollary \ref{Maincor} yields the complete solution for the averaging problems over the homogeneous varieties in three dimension.
\subsection{Definition of the averaging problem for finite fields}
We review the averaging problem over algebraic varieties in the finite field setting.
Let $V$ be an algebraic variety in $(\mathbb F_q^d, dx), d\geq 2,$ where $dx$ also denotes the normalized counting measure.
We also denotes by $d\sigma$ the normalized surface measure on $V.$
For $1\leq p, r\leq \infty,$ define $A(p\to r)$ as the smallest constant such that for every $f$ defined on $({\mathbb F_{q}^d},dx)$, we have
$$ \|f\ast d\sigma\|_{L^r({\mathbb F_{q}^d, dx})} \leq A(p\to r) \|f\|_{L^p({\mathbb F_{q}^d, dx})},$$
where we recall that $f\ast d\sigma(x)= \int_{V} f(x-y) d\sigma(y)= \frac{1}{|V|} \sum_{y\in V} f(x-y).$ Then, the averaging problem is to determine the exponents $1\leq p,r\leq \infty$ such that $ A(p\to r)\leq C $ for some constant $C>0$ independent of $q,$ the size of the underlying finite field $\mathbb F_q.$

\subsection{ Sharp boundedness of the averaging operator on homogeneous varieties in $\mathbb F_q^3$}
Now, let us consider the homogeneous variety $H$ in three dimension given by
\begin{equation}\label{Defh} H=\{x\in \mathbb F_q^3: P(x)=0\}, \end{equation}
where $P(x)\in \mathbb F_q[x_1,x_2,x_3]$ is a homogeneous polynomial.
The following theorem is our main theorem whose proof is based on applying  well-known harmonic analysis methods for the Euclidean case.
We shall prove our main theorem by adopting the arguments in \cite{CSW08}.
\begin{theorem}\label{mainaveraging} Let $H \subset \mathbb F_q^3$ be the homogeneous variety given as in (\ref{Defh}).
Assume that $|H|\sim q^2$ and $H$ does not contain any plane passing through the origin. Then, we have that
$ A(p\to r)\lesssim 1 $ if and only if  $(1/p, 1/r)$ is contained in the convex hull of the points $(0,0), (0,1),(1,1), (3/4, 1/4).$
\end{theorem}

\begin{remark} In the Euclidean case, it is well known that if $1\leq r < p \leq \infty,$ then $L^p-L^r$ estimate is impossible.
However, in the finite field setting, we shall see that it is always true that $R^*(p\to r)\lesssim 1$ for $1\leq r < p \leq \infty.$
Like the Euclidean case,  the main interest for finite fields will be also the case when $1\leq p \leq r \leq \infty.$
\end{remark}
\begin{proof} We prove Theorem \ref{mainaveraging}.\\
\textbf{$(\Longrightarrow )$} Suppose that $A(p\to r) \lesssim 1 $ for $1\leq p,r\leq \infty.$ Then, it must be true that
for every function $f$ on $({\mathbb F_q^3},dx),$
$$\|f\ast d\sigma\|_{L^r({\mathbb F_q^3, dx})} \lesssim \|f\|_{L^p({\mathbb F_q^3, dx})}.$$
In particular, this inequality  also holds when we take $f=\delta_0,$ where $\delta_0(x)=0$ if $x\neq (0,0,0)$ and $\delta_0(0,0,0)=1.$
Thus, we see that
\begin{equation}\label{proofN}
\|\delta_0\ast d\sigma\|_{L^r({\mathbb F_q^3, dx})} \lesssim \|\delta_0\|_{L^p({\mathbb F_q^3, dx})}.\end{equation}
Since $dx$ is the normalized counting measure, the right hand side is given by
\begin{equation}\label{test1} \|\delta_0\|_{L^p(\mathbb F_q^3, dx)}=q^{-\frac{3}{p}}.\end{equation}
To estimate the left hand side, we recall from (\ref{alex}) that $d\sigma(x)=q^3 |H|^{-1} H(x)~dx$ and notice that
$$(\delta_0\ast d\sigma)(x)= \frac{q^3}{|H|} (\delta_0\ast H)(x)=\frac{1}{|H|}\delta_{H}(x),$$
where $\delta_H(x)=1$ if $x\in H$, and $\delta_H(x)=0$ if $x\notin H.$  Thus, the left hand side in (\ref{proofN}) is given by
\begin{equation}\label{test2}
\|\delta_0\ast d\sigma\|_{L^r(\mathbb F_q^3, dx)} =q^{-\frac{3}{r}}|H|^{\frac{1-r}{r}}\sim q^{\frac{-2r-1}{r}},\end{equation}
where we also used the hypothesis that $|H|\sim q^2.$ Thus, from (\ref{proofN}), (\ref{test1}), and (\ref{test2}), it must be true that  \begin{equation}\label{kohline1}\frac{3}{p}\leq \frac{1}{r}+2.\end{equation}
By duality we also see that it must be true that
$$ \frac{3}{r'}\leq \frac{1}{p'}+2.$$
From this and (\ref{kohline1}), a simple calculation shows that $(1/p, 1/r)$ must be contained in the convex hull of  the points $(0,0), (0,1),(1,1),
(3/4, 1/4).$\\

\textbf{$(\Longleftarrow )$} We must show that $ A(p\to r)\lesssim 1 $ for all $1\leq p,r\leq \infty$ such that $(1/p,1/r)$ lies in the convex hull of
the points $(0,0), (0,1),(1,1), (3/4, 1/4).$ To do this, first we shall prove that for every function $f$ on $(\mathbb F_q^3, dx),$
 \begin{equation}\label{trivialone} \|f\ast d\sigma\|_{L^r({\mathbb F_{q}^3, dx})}\lesssim  \|f\|_{L^p({\mathbb F_{q}^3, dx})} \quad \mbox{if}~~  1\leq r \leq p\leq \infty .\end{equation}
Next, we shall prove that for every function $f$ on $(\mathbb F_q^3, dx),$
\begin{equation}\label{hardone}\|f\ast d\sigma\|_{L^4({\mathbb F_{q}^3, dx})}\lesssim \|f\|_{L^{\frac{4}{3}}({\mathbb F_{q}^3, dx})}.\end{equation}
Finally, interpolating (\ref{trivialone}) and (\ref{hardone}) shall give the complete proof.
Now, let us prove that $(\ref{trivialone})$ holds. Since $d\sigma$ is the normalized surface measure and $dx$ is the normalized counting measure, we see that both $d\sigma$ and $({\mathbb F_{q}^3},dx)$ have total mass $1$. It therefore follows from Young's inequality and H\"{o}lder's inequality that if $1\leq r \leq p\leq \infty$, then
\begin{equation} \|f\ast d\sigma\|_{L^r({\mathbb F_{q}^3, dx})}\leq \|f\|_{L^r({\mathbb F_{q}^3, dx})} \leq \|f\|_{L^p({\mathbb F_{q}^3, dx})}.\end{equation}
To complete the proof, it therefore suffices to show that the inequality (\ref{hardone}) holds.
As before, we consider  a function $K$ on $({\mathbb F_q^3}, dm)$ defined as $K=(d\sigma)^\vee -\delta_0.$
Note that for each $x\in (\mathbb F_q^3,dx)$, we have $\widehat{\delta_0}(x)=\int_{\mathbb F_q^3} \chi(-x\cdot m) \delta_0(m) dm=1,$
because $dm$ is the counting measure. Since $d\sigma= \widehat{K}+\widehat{\delta_0}= \widehat{K}+1$ and $\|f\ast 1\|_{L^{4}({\mathbb F_q^3},dx)}\lesssim \|f\|_{L^{\frac{4}{3}}({\mathbb F_q^3},dx)} \|1\|_{L^2({\mathbb F_q^3},dx)}=\|f\|_{L^{\frac{4}{3}}({\mathbb F_q^3},dx)}$ by Young's inequality, it is enough to show that for every $f$ on $({\mathbb F_q^3},dx),$ we have
$$\|f\ast \widehat{K}\|_{L^{4}({\mathbb F_q^3},dx)}\lesssim \|f\|_{L^{\frac{4}{3}}({\mathbb F_q^3},dx)}.$$
However, this inequality can be obtained by  interpolating the following two estimates:
\begin{equation}\label{first1}
\|f\ast \widehat{K}\|_{L^2({\mathbb F_q^3},dx)}\lesssim  q^{-1}\|f\|_{L^2({\mathbb F_q^3},dx)}
\end{equation}
and
\begin{equation}\label{second2}
\|f\ast \widehat{K}\|_{L^{\infty}({\mathbb F_q^3},dx)}\lesssim  q \|f\|_{L^1({\mathbb F_q^3},dx)}.
\end{equation}
Thus, it remains to prove that both (\ref{first1}) and (\ref{second2}) hold. From the definition of $K$ and (\ref{inversiondecay}) in Remark \ref{inversionremark}, it is clear that
$$ \|K\|_\infty \lesssim q^{-1}.$$
Thus, using this fact,  the inequality (\ref{first1}) follows from the Plancherel theorem.
On the other hand, the inequality (\ref{second2}) follows from Young's inequality and the observation that $\|\widehat{K}\|_{L^\infty({\mathbb F_q^3},dx)}\lesssim q.$ Thus, we complete the proof.
\end{proof}

\section{Erd\H os-Falconer distance problem for finite fields}
 Let $E, F\subset \mathbb F_q^d, d\geq 2.$ Given a polynomial $P(x)\in \mathbb F_q[x_1,\dots,x_d]$,  the generalized distance set $\Delta_P(E,F)$ can be defined by
 $$ \Delta_P(E,F) = \{P(x-y)\in \mathbb F_q : x\in E, y\in F\},$$
 throughout  this paper, we always assume the characteristic of $\mathbb F_q$ is sufficiently large (in the sense that comparing the degree of $P$).
In the finite field case,  the generalized Erd\H os distance problem is to determine the minimum cardinality of $\Delta_P(E,F)$ in terms of $|E|$ and $|F|.$ In the case when $E=F$ and $P(x)=x_1^2+x_2^2$, this problem was first introduced by Bourgain, Katz, and Tao \cite{BKT04}.
Using the discrete Fourier analytic machinery, Iosevich and Rudnev \cite{IR07} formulated this problem and obtained several interesting results.
For example, they proved the following:
\begin{theorem}\label{Iosevich} If $E\subset \mathbb F_q^d, d\geq 2,$ with $|E|\geq C q^{\frac{d}{2}}$ for $C>0$ sufficiently large, then we have
$$ |\Delta_P(E,E)|\gtrsim \min\left( q, |E|q^{-\frac{d-1}{2}}\right),$$
where $P(x)=x_1^2+\cdots+x_d^2.$
\end{theorem}
In addition, they addressed the Falconer distance problem for finite fields, which is the problem to determine the size of $E$ such that
$|\Delta_P(E,E)|\gtrsim q.$ Note that Theorem \ref{Iosevich} implies that if $P(x)=x_1^2+\cdots+x_d^2$ and $|E|\gtrsim q^{(d+1)/2}$, then $|\Delta_P(E,E)|\gtrsim q.$ Authors in \cite{HIKR10} observed that if the dimension $d\geq 3$ is odd, then the exponent $(d+1)/2$ gives the best possible result on the Falconer distance problem for finite fields. On the other hand, it has been conjectured that the exponent $d/2$ could be the best possible one  if the dimension $d\geq 2$ is even. In the case when $d=2,$ the sharp exponent $(d+1)/2$ for odd dimensions was improved by $4/3$
(see \cite{CEHIK09} or \cite{KS10}). From these facts, one may think that improving Theorem \ref{Iosevich} for even dimensions is only interesting.
However, we shall focus on the problem in odd dimensions. The main point we want to address is that if the dimension $d\geq 3$ is odd, then the condition in Theorem \ref{Iosevich},  $|E|\geq C q^{\frac{d}{2}}$, can be relaxed. On the other hand, the condition is necessary for even dimensions.
More generally, we consider the following conjecture.
\begin{conjecture}\label{KSconj} Let $P(x)= \sum_{j=1}^d a_j x_j^c \in \mathbb F_q[x_1,\dots,x_d]$ with $a_j\neq 0, c\geq 2.$
If $E, F \subset \mathbb F_q^d$ and $d\geq 3$ is odd, then we have
$$ |\Delta_P(E,F)| \gtrsim \min\left( q, q^{-\frac{d-1}{2}} \sqrt{|E||F|}\right).$$
\end{conjecture}

Authors in \cite{KoS10} proved that the conclusion in Conjecture \ref{KSconj} holds for all dimensions $d\geq 2$ if we assume that $|E||F|\geq C q^d$ for  a sufficiently large constant $C>0$ (see Corollary 3.5 in \cite{KoS10}). They also introduced a simple example to show that if the dimension $d$ is even, then  the assumption $|E||F|\geq C q^d$ is necessary. In addition, they pointed out that Conjecture \ref{KSconj} is true if $c=2.$
In this section, we shall prove that Conjecture \ref{KSconj} is true in the case when the dimension $d$ is three.
Observe that if Conjecture \ref{KSconj} is true, then the distance set has its nontrivial cardinality for $|E||F|\geq C q^{d-1}$ with $d$ odd.

\subsection{Main result for the Erd\H os-Falconer distance problem} In this subsection, we prove the following main theorem.
\begin{theorem} \label{mainthree}
In dimension three, Conjecture \ref{KSconj} is true.\end{theorem}

First, we derive a formula for proving Theorem \ref{mainthree}.
 Let $P(x)\in \mathbb F_q[x_1,\dots,x_d]$ be a polynomial with degree $\geq 2.$ For each $t\in \mathbb F_q,$ define a variety $H_t \subset \mathbb F_q^d,$ by the set
$$ H_t=\{x\in \mathbb F_q^d: P(x)=t\}.$$
Then, we have the following distance formula.
\begin{lemma}\label{stronglemma}Suppose that for every $m\in \mathbb F_q^d\setminus \{(0,\dots,0)\}$ and $t\in \mathbb F_q$, we have
\begin{equation}\label{SharpDecay} |\widehat{H_t}(m)|\lesssim q^{-\frac{d+1}{2}}.\end{equation}
Then, if $E, F \subset \mathbb F_q^d,$  then we have
$$ |\Delta_P(E,F)| \gtrsim \min\left( q, q^{-\frac{d-1}{2}} \sqrt{|E||F|}\right).$$
\end{lemma}

\begin{proof} First, we consider a counting function $\nu$ on $\mathbb F_q$, given by
$$ \nu(t)=|\{(x,y)\in E\times F: P(x-y)=t\}|=|\{(x,y)\in E\times F: x-y\in H_t\}|.$$
Recall that $\Delta_P(E,F)=\{P(x-y)\in \mathbb F_q: x\in E, y\in F\}$ and notice that
\begin{equation}\label{forit} |E||F|= \sum_{t\in \Delta_P(E,F)} \nu(t) \leq \left(\max_{t\in \mathbb F_q} \nu(t)\right)~ |\Delta_P(E,F)|.\end{equation}
Thus, the estimate for the upper bound of $\max_{t\in \mathbb F_q} \nu(t)$ is needed. For each $t\in \mathbb F_q,$ applying the Fourier inversion theorem (\ref{Finversion}) to the function $ H_t(x-y),$ and then  using the definition of the Fourier transform, we see that
$$\nu(t)= \sum_{x\in E, y\in F} H_t(x-y)= q^{2d} \sum_{m\in {\mathbb F_q^d}}\overline{\widehat{E}}(m) \widehat{F}(m) \widehat{H_t}(m).$$
Now, write $\nu(t)$ by
\begin{align*}\label{counting}
\nu(t)=&q^{2d} \overline{\widehat{E}}(0,\dots,0) \widehat{F}(0,\dots,0) \widehat{H_t}(0,\dots,0) +
q^{2d} \sum_{m\in {\mathbb F_q^d}\setminus \{(0,\dots,0)\}}\overline{\widehat{E}}(m) \widehat{F}(m) \widehat{H_t}(m)\\
=&\mbox{I} + \mbox{II}\nonumber. \end{align*}

From the definition of the Fourier transform and the Schwartz-Zippel lemma, it follows that
$$ |\mbox{I}|= \frac{1}{q^d} |E||F||H_t| \lesssim q^{-1}|E||F|.$$
On the other hand, our hypothesis (\ref{SharpDecay}) and the Cauchy-Schwarz inequality yield
$$|\mbox{II}|\lesssim q^{2d}q^{-\frac{d+1}{2}} \left(\sum_{m} \left|\overline{\widehat{E}}(m)\right|^2\right)^\frac{1}{2}
\left(\sum_{m} \left|\widehat{F}(m)\right|^2\right)^\frac{1}{2}.$$
Applying the Plancherel theorem (\ref{Plancherel}), we obtain
$$|\mbox{II}| \lesssim q^{\frac{d-1}{2}} |E|^{\frac{1}{2}} |F|^{\frac{1}{2}}.$$
Thus, it follows that
$$ \max_{t\in \mathbb F_q} \nu(t) \lesssim q^{-1}|E||F| + q^{\frac{d-1}{2}} |E|^{\frac{1}{2}} |F|^{\frac{1}{2}}.$$
From this fact and (\ref{forit}), a direct calculation completes the proof.
\end{proof}

It seems that the assumption (\ref{SharpDecay}) in Lemma \ref{stronglemma} is too strong. For example, if dimension $d$ is even, $H_t=\{x\in \mathbb F_q^d: x_1^c+\dots+x_d^c=t\}, c\geq 2,$ and $ u^c=-1$ for some $u\in \mathbb F_q,$ then  this case can not satisfy the assumption (\ref{SharpDecay}). This follows from a simple observation that if $E=F=\{ (t_1, ut_1, \dots, t_{d/2}, u t_{d/2}) \in \mathbb F_q^d: t_j\in \mathbb F_q \}$, then $|E|=|F|=q^{d/2} $ and $|\Delta_P(E,F)|=|\{0\}|=1,$ which does not satisfy the conclusion of Lemma \ref{stronglemma}. However, observe that if the dimension $d$ is odd, then the similar example does not exist. For this reason, Conjecture \ref{KSconj} looks true. In fact, the following lemma says that only $H_0$ in the previous example violates the assumption (\ref{SharpDecay}).

\begin{lemma}[4.4.19 in \cite{Co94}]\label{Todd} Let $P(x)=\sum\limits_{j=1}^d a_jx_j^s \in \mathbb F_q[x_1,\dots,x_d]$ with $s\geq 2, a_j\neq 0$ for all $j=1,\dots,d.$
In addition, assume that the characteristic of $\mathbb F_q$ is sufficiently large  so that  it does not divide $s.$ Then,
$$ |\widehat{H_t}(m)|=\frac{1}{q^d} \left|\sum_{x\in H_t} \chi(-x\cdot m)\right| \lesssim q^{-\frac{d+1}{2}} \quad \mbox{for all}~~ m\in \mathbb F_q^d\setminus \{(0,\dots,0)\}, t\in \mathbb F_q\setminus\{0\},$$
and
\begin{equation}\label{mainissue} |\widehat{H_0}(m)| \lesssim q^{-\frac{d}{2}} \quad \mbox{for all} ~~ m \in \mathbb F_q^d\setminus \{(0,\dots,0)\},\end{equation}
where $H_t =\{x\in \mathbb F_q^d: P(x)=t\}.$
\end{lemma}

\subsection{Complete proof of Theorem \ref{mainthree}} From Lemma \ref{Todd} and Lemma \ref{stronglemma}, it suffices to prove that for every $m\neq (0,0,0),$
$$ |\widehat{H_0}(m)|\lesssim q^{-2},$$
where $H_0=\{x\in \mathbb F_q^3: a_1x_1^c+a_2x_2^c+a_3x_3^c=0\}$ with $a_j\neq 0, j=1,2,3,  c\geq 2.$
From Corollary \ref{Maincor}, it is enough to show that $H_0$ does not contain any plane passing through the origin in $\mathbb F_q^3.$
By contradiction, assume that $H_0$ contains  a plane $\Pi_m =\{x\in \mathbb F_q^3: m\cdot x=0\}$ for some $m\neq (0,0,0).$
Without loss of generality, assume that $\Pi_m=\{x\in \mathbb F_q^3: x_3=m_1^\prime x_1 +m_2^\prime x_2\}$ for some $m_1^\prime, m_2^\prime \in \mathbb F_q.$ Then, $|H_0\cap \Pi_m|=|\Pi_m|=q^2.$ However, this is impossible if $q$ is sufficiently large. To see this, notice that
$$ |H_0\cap \Pi_m|=|\{(x_1,x_2, m_1^\prime x_1 +m_2^\prime x_2)\in \mathbb F_q^3: a_1x_1^c +a_2x_2^c+ a_3 (m_1^\prime x_1 +m_2^\prime x_2)^c=0\}|$$
$$=|\{(x_1,x_2)\in \mathbb F_q^2: a_1x_1^c +a_2x_2^c+ a_3 (m_1^\prime x_1 +m_2^\prime x_2)^c=0\}|\lesssim q,$$
where the last inequality follows from the Schwartz-Zippel lemma, because one can check that $a_1x_1^c +a_2x_2^c+a_3 (m_1^\prime x_1 +m_2^\prime x_2)^c$ is a nonzero polynomial for $c\geq 2$ and $a_j\neq 0 , j=1,2,3.$ Thus, we complete the proof of Theorem \ref{mainthree}.

\section{Note on problems for homogeneous varieties in higher odd dimensions}
We have seen that the good Fourier decay on homogeneous varieties makes a key role to study the problems in this paper.
We also have been mentioning that one could obtain good Fourier decay on homogeneous varieties in odd dimensions, but fail in even dimensions. For reader's convenience, we give the following explicit computations to indicate the estimates of Fourier transform over quadratic homogeneous varieties are different between odd and even dimensions.  Now let us see the following two examples.
First, suppose that $V=\{x\in \mathbb F_q^3: x_1^2+x_2^2+x_3^2=0\}.$
Then, for each $m\neq (0,0,0),$ we have
$$ \widehat{V}(m)= q^{-4} \sum_{s\neq 0} \prod_{j=1}^3 \sum_{x_j\in \mathbb F_q} \chi( sx_j^2-m_jx_j) .$$
Completing the square and making a change of variables, we observe that
$$\sum_{x_j\in \mathbb F_q} \chi( sx_j^2-m_jx_j) = \sum_{x_j\in \mathbb F_q} \chi(sx_j^2) \chi\left( \frac{ m_j^2}{-4s}\right)$$
$$= G \eta(s) \chi\left( \frac{ m_j^2}{-4s}\right),$$
where $G$ denotes the Gauss sum, $\eta$ denotes the quadratic character, and we use the fact that $\sum_{x_j\in \mathbb F_q} \chi(sx_j^2) = G \eta(s).$
Thus, we see that for $m\neq (0,0,0),$
$$\widehat{V}(m)= q^{-4} G^3 \sum_{s\neq 0} \eta^3(s) \chi\left( \frac{ m_1^2+m_2^2+m_3^2}{-4s}\right).$$
Since $\eta$ is the quadratic character, $ \eta^3=\eta,$ and so the sum over $s\neq 0$ is a Sali\'e sum \cite{Sa32} which is always $\lesssim q^{1/2}.$ Thus, we get a good Fourier decay on $V$ for all $m\neq (0,0,0).$ Namely, we have for $m\neq (0,0,0),$
$$|\widehat{V}(m)|\lesssim q^{-2}= q^{-\frac{d+1}{2}},$$ which is what we want.\\

However, now consider  $V=\{x\in \mathbb F_q^4: x_1^2+x_2^2+x_3^2+x_4^2=0\}.$
Using above method, we see that for each $m\neq (0,0,0,0)$
$$ \widehat{V}(m)= q^{-5} G^4 \sum_{s\neq 0} \eta^4(s) \chi\left( \frac{ m_1^2+m_2^2+m_3^2+m_4^2}{-4s}\right).$$
Since $\eta^4=1$, the sum over $s\neq 0 $ is  $(q-1)$ if $m_1^2+m_2^2+m_3^2+m_4^2=0.$ Thus, for some $m\neq (0,0,0,0),$ we have
$$ |\widehat{V}(m)|\sim q^{-2}= q^{-\frac{d}{2}},$$ which is worse than $q^{-\frac{5}{2}}=q^{-\frac{d+1}{2}}.$ Therefore, the question we want to address first is whether the estimate (\ref{mainissue}) in Lemma \ref{Todd} can be improved in all odd dimensions.
If the dimension is even, then  above estimates say that we can not expect to improve it (at least for the case $s=2$).
However, we already observed that in three dimension the estimate (\ref{mainissue}) can be improved to $|\widehat{H_0}(m)| \lesssim q^{-\frac{d+1}{2}}=q^{-2}$ for $m\neq (0,0,0).$ From these facts, one may have the following question.
\begin{question}\label{trueconjecture} Let $P(x)=\sum\limits_{j=1}^d a_jx_j^s \in \mathbb F_q[x_1,\dots,x_d]$ with $s\geq 2, a_j\neq 0$ for all $j=1,\dots,d.$
If we assume that the characteristic of $\mathbb F_q$ is sufficiently large and the dimension $d\geq 3$ is odd, then
does the following conclusion  always hold?
$$ |\widehat{H}(m)|\lesssim q^{-\frac{d+1}{2}} \quad \mbox{for all}~~ m\in \mathbb F_q^d \setminus \{(0,\dots,0)\},$$
where $H=\{x\in \mathbb F_q^d: P(x)=0\}.$
\end{question}

If the answer for Question \ref{trueconjecture} would be positive, then this would yield the Tomas-Stein exponent for the extension problem related to diagonal polynomials in odd dimensions. Moreover, the averaging problem on homogeneous varieties in odd dimensions would be completely understood.
 As a trial to find the answer for Question \ref{trueconjecture}, one may invoke some powerful results from algebraic geometry such as \cite{Ka99} and \cite{Co94}. However, it seems that such theorems do not explain  that the Fourier decays of homogeneous varieties in odd dimensions are better than in even dimensions. We also remark that if $P$ is not quadratic, it is also not clear that if the Fourier decay will be distinguished between odd dimensions and even dimensions.  We close this paper with a desire to see the answer for Question \ref{trueconjecture} in the near future.\\\\

{\bf Acknowledgment :} The authors would like to thank Nicholas Katz and Igor Shparlinski for some helpful discussions about the exponential sums.


\begin{thebibliography}{7}

\bibitem{BKT04} J.~Bourgain, N.~ Katz, and T.~ Tao,  \emph{A sum-product estimate in finite fields, and applications}, Geom. Funct. Anal. 14 (2004), 27--57.

\bibitem{CSW08} A.~Carbery, B.~Stones, and J.~Wright,  \emph{Averages in vector spaces over finite fields,} Math. Proc. Camb. Phil. Soc. (2008), 144, 13, 13--27.


\bibitem{CEHIK09} J. Chapman, M. Erdo\~{g}an, D. Hart, A. Iosevich, and D. Koh, \emph{Pinned distance sets, Wolff's exponent in finite fields and sum-product estimates},  arXiv:0903.4218v2, (2009).

\bibitem{Co94} T.~ Cochrane, \emph{Exponential sums and the distribution of solutions of congruences}, Inst. of Math., Academia Sinica, Taipei, (1994).





\bibitem{Er05} M. Erdo\~{g}an, \emph{ A bilinear Fourier extension theorem and applications to the distance set problem,}
Internat. Math. Res. Notices 23 (2005), 1411-1425.

\bibitem{Er46} P.~ Erd\H os, \emph{On sets of distances of $n$ points}, Amer. Math. Monthly 53, (1946), 248--250.

\bibitem{Fa85} K. Falconer, \emph{On the Hausdorff dimensions of distance sets,} Mathematika, 32 (1985), 206--212.

\bibitem{HIKR10} D. Hart, A. Iosevich, D. Koh and M. Rudnev, \emph{ Averages over hyperplanes, sum-product theory in vector spaces over finite fields and the Erd\"os-Falconer distance conjecture}, Trans. Amer. Math. Soc. (2010) To appear.



\bibitem{IK08} A. Iosevich and D. Koh, \emph{Extension theorems for the Fourier transform associated with non-degenerate quadratic surfaces in vector spaces over finite fields}, Illinois J. of Mathematics, Volume 52, Number 2, Summer (2008), 611--628.



\bibitem{IR07} A. Iosevich and M. Rudnev, \emph{ Erd\"{o}s distance problem in vector spaces over finite fields},
 Trans. Amer. Math. Soc. 359 (2007), 6127-6142.

\bibitem{IS96} A.~Iosevich and E.~Sawyer, \emph{Sharp $L^p-L^q$ estimates for a class of averaging operators,} Ann. Inst. Fourier, Grenoble, 46, 5 (1996), 1359--1384.

\bibitem{Ka99} N.~ Katz, \emph{ Estimates for ``singular" exponential sums,} Internat. Math. Res. Not. IMRN (1999), no. 16, 875--899.


\bibitem{Ko09} D.~ Koh , \emph {Extension and averaging operators for finite fields }, preprint (2009), arxiv.org.

\bibitem{KS10} D. ~Koh and C.~ Shen, \emph {Sharp extension theorems and Falconer distance problems for algebraic curves in two dimensional vector spaces over finite fields }, preprint (2010), arxiv.org.

\bibitem{KoS10} D. ~Koh and C.~ Shen, \emph{The generalized Erd\"{o}s-Falconer distance problems in vector spaces over finite fields}, preprint, arxiv.org.

\bibitem{KT04} N.~ Katz and G.~ Tardos, \emph {A new entropy inequality for the Erd\"os distance problem},
Contemp. Math. \textbf{342}, Towards a theory of geometric graphs, 119-126, Amer. Math. Soc., Providence, RI (2004).

\bibitem{Li73} W.~Littman, \emph{$L^p-L^q$ estimates for singular integral operators,} Proc. Symp. Pure Math., 23 (1973), 479--481.

\bibitem{LN93} R. Lidl and H. Niederreiter, \emph{ Finite fields,} Cambridge University Press, (1993).


\bibitem{MT04} G. Mockenhaupt, and T. Tao, \emph{Restriction and Kakeya phenomena for finite fields}, Duke Math. J. 121(2004), no. 1, 35--74.

\bibitem{Sa32} H, Sali\'e, \emph{\" Uber die Kloostermanschen summen S(u,v;q)}, Math.Z.34, (1932),91-109.





\bibitem{St71} E.~M.~Stein, \emph{$L^p$ boundedness of certain convolution operators,} Bull. Amer. Math. Soc., 77 (1971), 404--405.

\bibitem{St78}
E. Stein, \emph{ Some problems in harmonic analysis, Harmonic analysis in Euclidean spaces} ( Proc. Sympos. Pure Math., Williams Coll., Williamstown, Mass.), (1978), Part 1, pp. 3-20.

\bibitem{St70} R. Strichartz, \emph{Convolutions with kernels having singularities on the sphere,} Trans. Amer. Math. Soc., 148 (1970), 461--471.

\bibitem{SolTo01} J.~ Solymosi and C.~ T\'{o}th, \emph{Distinct distances in the plane}, Discrete Comput. Geom. {\textbf 25} (2001), no.~4, 629--634.

\bibitem{SV04} J.~ Solymosi and V.~ Vu, \emph{Distinct distances in high dimensional homogeneous sets} in: Towards a Theory of Geometric Graphs (J. Pach, ed.), Contemporary Mathematics, vol. 342, Amer. Math. Soc. (2004).

\bibitem{SV05} J.~ Solymosi and V.~ Vu, \emph{Near  optimal bounds for the number of distinct distances in high dimensions}, Combinatorica, Vol 28, no 1 (2008), 113--125.

\bibitem{Ta03}
T. Tao, \emph{A sharp bilinear restriction estimate for paraboloids}, Geom. Funct. Anal. {\bf 13} (2003), 1359--1384.

\bibitem{Ta04}
 T. Tao, \emph{Some recent progress on the restriction conjecture}, Fourier
analysis and convexity,  217--243, Appl. Numer. Harmon. Anal., Birkh\"auser Boston, Boston, MA, (2004).

\bibitem{We48} A. Weil, \emph{On some exponential sums,} Proc. Nat. Acad. Sci. U.S.A. {\bf 34} (1948), 204--207.


\bibitem{Wo99} T. Wolff, \emph{Decay of circular means of Fourier transforms of measures,} Internat. Math. Res. Notices 1999, 547--567.













\end{thebibliography}
\end{document}